\documentclass[10pt,a4paper]{article}
\usepackage[utf8]{inputenc}
\usepackage[T1]{fontenc}

\usepackage{tikz}
\usetikzlibrary{shapes,arrows,calc, fit, decorations}
\usetikzlibrary{decorations.pathreplacing,decorations.markings}
\usetikzlibrary{fit}

\usepackage{float}
\usepackage{a4wide}
\usepackage{amssymb}
\usepackage[colorlinks=true,citecolor=black,linkcolor=black,urlcolor=blue]{hyperref}

\usepackage{pifont}

\usepackage[makeroom]{cancel}

\usepackage{array}
\newcolumntype{M}[1]{>{\centering\arraybackslash}m{#1}}
\newcolumntype{N}{@{}m{0pt}@{}}

\usepackage{caption}
\usepackage{subcaption} 

\usepackage{amsthm}
\usepackage{amsmath}
\usepackage{mathrsfs}

\usepackage{wrapfig}

\usepackage{stackengine}

\newtheorem{theorem}{Theorem}
\newtheorem{corollary}[theorem]{Corollary}
\newtheorem{proposition}[theorem]{Proposition}
\newtheorem{lemma}[theorem]{Lemma}

\def\abs#1{\left| #1 \right|}
\def\paren#1{\left( #1 \right)}
\def\acc#1{\left\{ #1 \right\}}

\def\ceil#1{\left\lceil #1 \right\rceil}

\renewcommand{\le}{\leqslant}
\renewcommand{\ge}{\geqslant}

\DeclareMathOperator{\mad}{\mathrm{mad}}

\title{Homomorphisms of planar $(m, n)$-colored-mixed graphs to planar targets}

\date{}
\author{Fabien Jacques and Pascal Ochem\thanks{ This work is supported by the ANR project HOSIGRA (ANR-17-CE40-0022)}\\ LIRMM, Universit\'e de Montpellier, and CNRS. France}

\begin{document}

\maketitle

\begin{abstract}
An $(m, n)$-colored-mixed graph $G=(V, A_1, A_2,\cdots, A_m, E_1, E_2,\cdots, E_n)$ is a graph having $m$ colors of arcs and $n$ colors of edges.
We do not allow two arcs or edges to have the same endpoints. A homomorphism from an $(m,n)$-colored-mixed graph $G$
to another $(m, n)$-colored-mixed graph $H$ is a morphism $\varphi:V(G)\rightarrow V(H)$ such that each edge (resp. arc)
of $G$ is mapped to an edge (resp. arc) of $H$ of the same color (and orientation).
An $(m,n)$-colored-mixed graph $T$ is said to be $P_g^{(m, n)}$-universal if every graph in $P_g^{(m, n)}$ (the planar $(m, n)$-colored-mixed graphs with girth at least $g$)
admits a homomorphism to $T$.

We show that planar $P_g^{(m, n)}$-universal graphs do not exist for $2m+n\ge3$ (and any value of $g$) and
find a minimal (in the number vertices) planar $P_g^{(m, n)}$-universal graphs in the other cases. 
\end{abstract}

\section{Introduction}

The concept of homomorphisms of $(m, n)$-colored-mixed graph was introduced by J. Nes\v{e}t\v{r}il and A. Raspaud~\cite{MNCM}
in order to generalize homomorphisms of $k$-edge-colored graphs and oriented graphs.

An \emph{$(m, n)$-colored-mixed graph} $G=(V, A_1, A_2,\cdots, A_m, E_1, E_2,\cdots, E_n)$ is a graph having $m$ colors of arcs and $n$ colors of edges.
We do not allow two arcs or edges to have the same endpoints.
The case $m=0$ and $n=1$ corresponds to simple graphs, $m=1$ and $n=0$ to oriented graphs and $m=0$ and $n=k$ to $k$-edge-colored graphs. For the case $m=0$ and $n = 2$
($2$-edge-colored graphs) we refer to the two types of edges as \emph{blue} and \emph{red} edges.

A \emph{homomorphism} from an $(m, n)$-colored-mixed graph $G$ to another $(m, n)$-colored-mixed graph $H$ is a mapping $\varphi:V(G) \rightarrow V(H)$
such that every edge (resp. arc) of $G$ is mapped to an edge (resp. arc) of $H$ of the same color (and orientation).
If $G$ admits a homomorphism to $H$, we say that $G$ is \emph{$H$-colorable} since this homomorphism can be seen as a coloring of the vertices of $G$
using the vertices of $H$ as colors. The edges and arcs of $H$ (and their colors) give us the rules that this coloring must follow.
Given a class of graphs $\mathcal{C}$, a graph is \emph{$\mathcal{C}$-universal} if for every graph $G \in \mathcal{C}$ is $H$-colorable.
The class $P_g^{(m, n)}$ contains every planar $(m, n)$-colored-mixed graph with girth at least $g$.

In this paper, we consider some planar $P_g^{(m, n)}$-universal graphs with $k$ vertices.
They are depicted in Figures~\ref{fig:t_oriented} and~\ref{fig:t_2edge}.
The known results about this topic are as follows.

\begin{theorem}\label{thm:known}{\ }
\begin{enumerate}
\item $K_4$ is a planar $P^{(0,1)}_3$-universal graph. This is the four color theorem.
\item $K_3$ is a planar $P^{(0,1)}_4$-universal graph. This is Grötzsch's Theorem \cite{grotzsch}.
\item $\overrightarrow{C_6^2}$ is a planar $P_{16}^{(1,0)}$-universal graph~\cite{P10}.
\end{enumerate}
\end{theorem}

Our first result shows that, in addition to the case of $(0,1)$-graphs covered by Theorems~\ref{thm:known}.1 and~\ref{thm:known}.2,
our topic is actually restricted to the cases of oriented graphs (i.e., $(m,n)=(1,0)$) and 2-edge-colored graphs (i.e., $(m,n)=(0,2)$).

\begin{theorem}\label{thm:Pmn}
For every $g\ge3$, there exists no planar $P_g^{(m,n)}$-universal graph if $2m+n\ge3$.
\end{theorem}

As Theorems~\ref{thm:known}.1 and~\ref{thm:known}.2 show for $(0,1)$-graphs, there might exist a trade-off between minimizing the girth $g$ and the number
of vertices of the universal graph, for a fixed pair $(m,n)$.
For oriented graphs, Theorem~\ref{thm:known}.3 tries to minimize the girth.
For oriented graphs and 2-edge-colored graphs, we choose instead to minimize the number of vertices of the universal graph.

\begin{theorem}\label{thm:positive}{\ }
\begin{enumerate}
\item $\overrightarrow{T_5}$ is a planar $P_{28}^{(1,0)}$-universal graph on 5 vertices.
\item $T_6$ is a planar $P_{22}^{(0, 2)}$-universal graph on 6 vertices.
\end{enumerate}
\end{theorem}

The following results shows that Theorem~\ref{thm:positive} is optimal in terms of the number of vertices of the universal graph.

\begin{theorem}\label{thm:negative}{\ }
\begin{enumerate}
\item For every $g\ge3$, there exists an oriented bipartite cactus graph (i.e., $K_4^-$ minor-free graph) with girth at least $g$ and oriented chromatic number at least 5.
\item For every $g\ge3$, there exists a 2-edge-colored bipartite outerplanar graph (i.e., $(K_4^-,K_{2,3})$ minor-free graph) with girth at least $g$ that does not map to a planar graph with at most 5 vertices.
\end{enumerate}
\end{theorem}

Most probably, Theorem~\ref{thm:positive} is not optimal in terms of girth. The following constructions give lower bounds on the girth.

\begin{theorem}\label{thm:ce}{\ }
\begin{enumerate}
\item There exists an oriented bipartite 2-outerplanar graph with girth $14$ that does not map to $\overrightarrow{T_5}$.
\item There exists a 2-edge-colored planar graph with girth $11$ that does not map to $T_6$.
\item There exists a 2-edge-colored bipartite planar graph with girth $10$ that does not map to $T_6$.
\end{enumerate}
\end{theorem}

\begin{figure}[H]
\begin{minipage}{0.5\textwidth}
\begin{center}
 \includegraphics{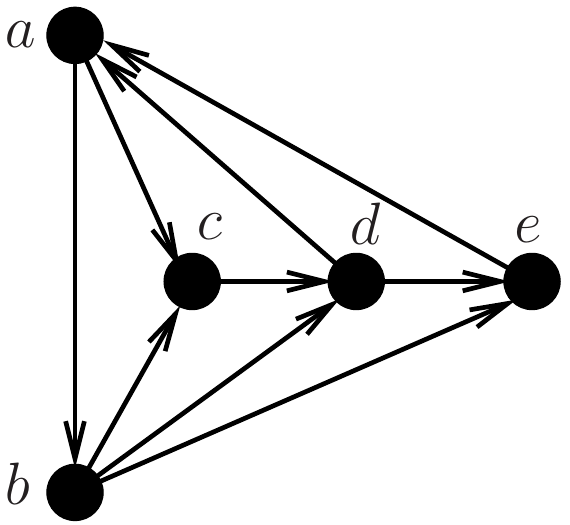}
 \caption{The $P_{28}^{(1,0)}$-universal graph overrightarrow$(T_5)$.\label{fig:t_oriented}}
 %\caption{The $P_{28}^{(1,0)}$-universal graph $\overrightarrow{T_5}$.\label{fig:t_oriented}}
\end{center}
\end{minipage}\hfill
\begin{minipage}{0.5\textwidth}
\begin{center}
 \includegraphics{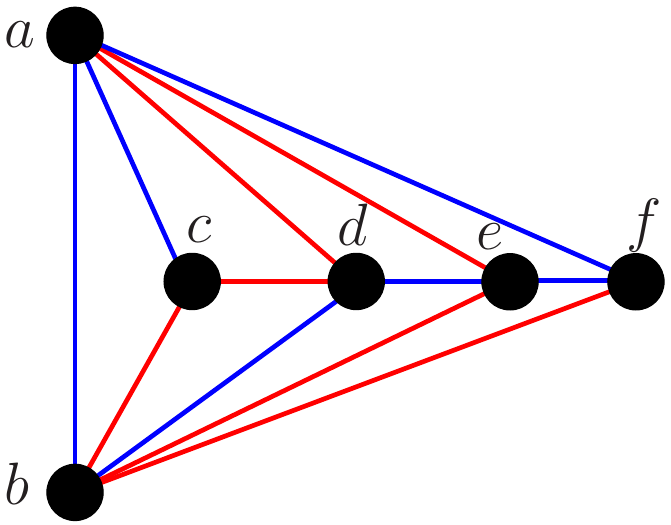}
 \caption{The $P_{22}^{(0,2)}$-universal graph $T_6$.\label{fig:t_2edge}}
\end{center}
\end{minipage}
\end{figure}

Next, we obtain the following complexity dichotomies:

\begin{theorem}\label{thm:NPC}{\ }
\begin{enumerate}
\item For any fixed girth $g\ge 3$, either every graph in $P_g^{(1,0)}$ maps to $\overrightarrow{T_5}$ or it is NP-complete
to decide whether a graph in $P_g^{(1,0)}$ maps to $\overrightarrow{T_5}$.
Either every bipartite graph in $P_g^{(1,0)}$ maps to $\overrightarrow{T_5}$ or it is NP-complete to decide whether a bipartite graph in $P_g^{(1,0)}$ maps to $\overrightarrow{T_5}$.
\item Either every graph in $P_g^{(0,2)}$ maps to $T_6$ or it is NP-complete to decide whether a graph in $P_g^{(1,0)}$ maps to $T_6$.
Either every bipartite graph in $P_g^{(0,2)}$ maps to $T_6$ or it is NP-complete to decide whether a bipartite graph in $P_g^{(1,0)}$ maps to $T_6$.
\end{enumerate}
\end{theorem}

Finally, we can use Theorem~\ref{thm:NPC} with the non-colorable graphs in Theorem~\ref{thm:ce}.

\begin{corollary}\label{cor:cor}{\ }
\begin{enumerate}
\item Deciding whether a bipartite graph in $P_{14}^{(1,0)}$ maps to $\overrightarrow{T_5}$ is NP-complete.
\item Deciding whether a graph in $P_{11}^{(0,2)}$ maps to $T_6$ is NP-complete.
\item Deciding whether a bipartite graph in $P_{10}^{(0,2)}$ maps to $T_6$ is NP-complete.
\end{enumerate}
\end{corollary}

A 2-edge-colored path or cycle is said to be \emph{alternating} if any two adjacent edges have distinct colors.

\begin{proposition}[folklore]\label{prop:3n-6}{\ }
\begin{itemize}
\item Every planar simple graph on $n$ vertices has at most $3n-6$ edges.
\item Every planar simple graph satisfies $(\mad(G)-2)\cdot(g(G)-2)<4$.
\end{itemize}
\end{proposition}

\section{Proof of Theorem~\ref{thm:positive}}
We use the discharging method for both results in Theorem~\ref{thm:positive}. The following lemma will handle the discharging part.
We call a vertex of degree $n$ an $n$-vertex and a vertex of degree at least $n$ an $n^+$-vertex.
If there is a path made only of $2$-vertices linking two vertices $u$ and $v$, we say that $v$ is a weak-neighbor of $u$.
If $v$ is a neighbor of $u$, we also say that $v$ is a weak-neighbor of $u$. We call a (weak-)neighbor of degree $n$ an $n$-(weak-)neighbor.

\begin{lemma}\label{lem:discharge}
Let $k$ be a non-negative integer.
Let $G$ be a graph with minimum degree 2 such that every 3-vertex has at most $k$ 2-weak-neighbors and every path contains at most $\tfrac{k+1}2$ consecutive 2-vertices.
Then $\mad(G)\ge2+\tfrac2{k+2}$. In particular, $G$ cannot be a planar graph with girth at least $2k+6$.
\end{lemma}

\begin{proof}
Let $G$ be as stated. Every vertex has an initial charge equal to its degree. Every $3^+$-vertex gives $\tfrac1{k+2}$ to each of its 2-weak-neighbors.
Let us check that the final charge $ch(v)$ of every vertex $v$ is at least $2+\tfrac2{k+2}$.
\begin{itemize}
 \item If $d(v)=2$, then $v$ receives $\tfrac1{k+2}$ from both of its 3-weak-neighbors. Thus $ch(v)=2+\tfrac2{k+2}$.
 \item If $d(v)=3$, then $v$ gives $\tfrac1{k+2}$ to each of its 2-weak-neighbors. Thus $ch(v)\ge3-\tfrac{k}{k+2}=2+\tfrac2{k+2}$.
 \item If $d(v)=d\ge4$, then $v$ has at most $\tfrac{k+1}2$ 2-weak-neighbors in each of the $d$ incident paths.
 Thus $ch(v)\ge d-d\paren{\tfrac{k+1}2}\paren{\tfrac1{k+2}}=\tfrac d2\paren{1+\tfrac1{k+2}}\ge2+\tfrac2{k+2}$.
\end{itemize}
This implies that $mad(G)\ge2+\frac2{k+2}$.
Finally, if $G$ is planar, then the girth of $G$ cannot be at least $2k+6$, since otherwise $(\mad(G)-2)\cdot(g(G)-2)\ge\paren{2+\tfrac2{k+2}-2}\paren{2k+6-2}=\paren{\tfrac2{k+2}}\paren{2k+4}=4$, which contradicts Proposition~\ref{prop:3n-6}.
\end{proof}

%\begin{proposition}\label{prop:madgirth}
%\end{proposition}

\subsection{Proof of Theorem~\ref{thm:positive}.1}
We prove that the oriented planar graph $\overrightarrow{T_5}$ on 5 vertices from Figure~\ref{fig:t_oriented} is $P_{28}^{(1,0)}$-universal by contradiction.
Assume that $G$ is an oriented planar graphs with girth at least $28$ that does not admit a homomorphism to $\overrightarrow{T_5}$
and is minimal with respect to the number of vertices.
By minimality, $G$ cannot contain a vertex $v$ with degree at most one since a $\overrightarrow{T_5}$-coloring of $G-v$ can be extended to $G$.
Similarly, $G$ does not contain the following configurations.

\begin{itemize}
\item A path with 6 consecutive 2-vertices.
\item A $3$-vertex with at least 12 2-weak-neighbors. 
\end{itemize}

Suppose that $G$ contains a path $u_0u_1u_2u_3u_4u_5u_6u_7$ such that the degree of $u_i$ is two for $1\le i\le6$.
By minimality of $G$, $G-{u_1,u_2,u_3,u_4,u_5,u_6}$ admits a $\overrightarrow{T_5}$-coloring $\varphi$.
We checked on a computer that for any $\varphi(v_0)$ and $\varphi(v_6)$ in $V\paren{\overrightarrow{T_5}}$
and every possible orientation of the 7 arcs $u_iu_{i+1}$, we can always extend $\varphi$ into a $\overrightarrow{T_5}$-coloring of $G$, a contradiction.

Suppose that $G$ contains a 3-vertex $v$ with at least 12 2-weak-neighbors. Let $u_1$, $u_2$, $u_3$ be the $3^+$-weak-neighbors of $v$
and let $l_i$ be the number of common 2-weak-neighbors of $v$ and $u_i$, i.e., $2$-vertices on the path between $v$ and $l_i$.
Without loss of generality and by the previous discussion, we have $5\ge l_1\ge l_2\ge l_3$ and $l_1+l_2+l_3\ge12$.
So we have to consider the following cases:
\begin{itemize}
\item\textbf{Case 1:} $l_1=5$, $l_2=5$, $l_3=2$.
\item\textbf{Case 2:} $l_1=5$, $l_2=4$, $l_3=3$.
\item\textbf{Case 3:} $l_1=4$, $l_2=4$, $l_3=4$.
\end{itemize}

By minimality, the graph $G'$ obtained from $G$ by removing $v$ and its 2-weak-neighbors admits a $\overrightarrow{T_5}$-coloring $\varphi$.
Let us show that in all three cases, we can extend $\varphi$ into a $\overrightarrow{T_5}$-coloring of $G$ to get a contradiction.

With an extensive search on a computer we found that if a vertex $v$ is connected to a vertex $u$ colored in $\varphi(u)$ by a path
made of $l$ 2-vertices ($0\le l\le5$) then $v$ can be colored in:

\begin{itemize}
\item at least 1 color if $l=0$,
\item at least 2 colors if $l=1$,
\item at least 2 colors if $l=2$ (the sets $\acc{c, d, e}$ and $\acc{b, c, d}$ are the only sets of size 3 that can be forbidden from $v$),
\item at least 3 colors if $l=3$, 
\item at least 4 colors if $l=4$ and
\item at least 4 colors if $l=5$ (only the sets $\acc{b}$, $\acc{c}$, and $\acc{e}$ can be forbidden from $v$).
\end{itemize}

In Case 1, $u_3$ forbids at most 3 colors from $v$ since $l_3=2$. If it forbids less than $3$ colors,
we will be able to find a color for $v$ since $u_1$ and $u_2$ forbid at most 1 color from $v$. The only sets of 3 colors that $u_3$ can forbid are $\acc{b,c,d}$ and $\acc{c, d, e}$.
Since $u_1$ and $u_2$ can each only forbid $b$, $c$ or $e$, we can always find a color for $v$.

In Case 2, $u_1$ and $u_2$ each forbid at most one color and $u_3$ forbids at most $2$ colors so there remains at least one color for $v$.

In Case 3, $u_1$, $u_2$, and $u_3$ each forbid at most one color, so there remains at least two colors for $v$.

We can always extend $\varphi$ into a $\overrightarrow{T_5}$-coloring of $G$, a contradiction.

So $G$ contains at most 5 consecutive 2-vertices and every 3-vertex has at most 11 2-weak-neighbors.
Using Lemma~\ref{lem:discharge} with $k=11$ contradicts the fact that the girth of $G$ is at least 28.

\subsection{Proof of Theorem~\ref{thm:positive}.2}
We prove that the 2-edge-colored planar graph $T_6$ on 6 vertices from Figure~\ref{fig:t_2edge} is $P_{22}^{(0,2)}$-universal by contradiction.
Assume that $G$ is a 2-edge-colored planar graphs with girth at least $22$ that does not admit a homomorphism to $T_6$ and is minimal with respect to the number of vertices.
By minimality, $G$ cannot contain a vertex $v$ with degree at most one since a $T_6$-coloring of $G-v$ can be extended to $G$.
Similarly, $G$ does not contain the following configurations.

\begin{itemize}
\item A path with 5 consecutive 2-vertices.
\item A $3$-vertex with at least 9 2-weak-neighbors. 
\end{itemize}

Suppose that $G$ contains a path $u_0u_1u_2u_3u_4u_5u_6$ such that the degree of $u_i$ is two for $1\le i\le5$.
By minimality of $G$, $G-{u_1, u_2, u_3, u_4, u_5}$ admits a $T_6$-coloring $\varphi$.
We checked on a computer that for any $\varphi(v_0)$ and $\varphi(v_6)$ in $V(T)$ and every possible colors of the 6 edges $u_iu_{i+1}$,
we can always extend $\varphi$ into a $T_6$-coloring of $G$, a contradiction.

Suppose that $G$ contains a 3-vertex $v$ with at least 9 2-weak-neighbors. Let $u_1$, $u_2$, $u_3$ be the $3^+$-weak-neighbors of $v$
and let $l_i$ be the number of common 2-weak-neighbors of $v$ and $u_i$, i.e., $2$-vertices on the path between $v$ and $l_i$.
Without loss of generality and by the previous discussion, we have $4\ge l_1\ge l_2\ge l_3$ and $l_1+l_2+l_3\ge 9$. So we have to consider the following cases:

\begin{itemize}
\item\textbf{Case 1:} $l_1=3$, $l_2=3$, $l_3=3$.
\item\textbf{Case 2:} $l_1=4$, $l_2=3$, $l_3=2$.
\item\textbf{Case 3:} $l_1=4$, $l_2=4$, $l_3=1$.
\end{itemize}

By minimality of $G$, the graph $G'$ obtained from $G$ by removing $v$ and its 2-weak-neighbors admits a $T_6$-coloring $\varphi$.
Let us show that in all three cases, we can extend $\varphi$ into a $T_6$-coloring of $G$ to get a contradiction.

With an extensive search on a computer we found that if a vertex $v$ is connected to a vertex $u$ colored in $\varphi(u)$
by a path $P$ made of $l$ 2-vertices ($0\le l\le 4$) then $v$ can be colored in:

\begin{itemize}
\item at least 1 color if $l=0$ (the sets ${a, c, d, e, f}$ and ${b, c, d, e, f}$ of colors are the only sets of size 5 that can be forbidden from $v$
for some $\varphi(u)\in T$ and edge-colors on $P$),
\item at least 2 colors if $l=1$ (the sets ${a, b, c, f}$ and ${b, c, e, f}$ are the only sets of size 4 that can be forbidden from $v$),
\item at least 3 colors if $l=2$ (the sets ${b, c, f}$, ${c, e, f}$ and ${d, e, f}$ are the only sets of size 3 that can be forbidden from $v$),
\item at least 4 colors if $l=3$ (the set ${c, b}$ is the only set of size 2 that can be forbidden from $v$), and
\item at least 5 colors if $l=4$ (the sets ${c}$ and ${f}$ are the only sets of size 1 that can be forbidden from $v$).
\end{itemize}

Suppose that we are in Case 1. Vertices $u_1$, $u_2$, and $u_3$ each forbid at most 2 colors from $v$ since $l_1=l_2=l_3=3$.
Suppose that $u_1$ forbids 2 colors. It has to forbid colors $c$ and $f$ (since it is the only pair of colors that can be forbidden by a path made of 3 2-vertices).
If $u_2$ or $u_3$ also forbids 2 colors, they will forbid the exact same pair of colors. We can therefore assume that they each forbid 1 color from $v$.
There are 6 available colors in $T_6$, so we can always find a color for $v$ and extend $\varphi$ to a $T_6$-coloring of $G$, a contradiction.
We proceed similarly for the other two cases.

So $G$ contains at most 4 consecutive 2-vertices and every 3-vertex has at most 8 2-weak-neighbors.
Then Lemma~\ref{lem:discharge} with $k=8$ contradicts the fact that the girth of $G$ is at least 22.

\section{Proof of Theorem~\ref{thm:negative}.1}
We construct an oriented bipartite cactus graph with girth at least $g$ and oriented chromatic number at least 5. Let $g'$ be such that $g'\ge g$ and $g'\equiv4\pmod{6}$.
Consider a circuit $v_1,\cdots,v_{g'}$. Clearly, the oriented chromatic number of this circuit is 4 and the only tournament on 4 vertices it can map to
is the tournament $\overrightarrow{T_4}$ induced by the vertices $a$, $b$, $c$, and $d$ in $\overrightarrow{T_5}$.
Now we consider the cycle $C=w_1,\cdots,w_{g'}$ containing the arcs $w_{2i-1}w_{2i}$ with $1\le i\le g'/2$, $w_{2i+1}w_{2i}$ with $1\le i\le g'/2-1$, and $w_{g'}w_1$.

Suppose for contradiction that $C$ admits a homomorphism $\varphi$ such that $\varphi(w_1)=d$.
This implies that $\varphi(w_2)=a$, $\varphi(w_3)=d$, $\varphi(w_4)=a$, and so on until $\varphi(w_{g'})=a$.
Since $\varphi(w_{g'})=a$ and $\varphi(w_1)=d$, $w_{g'}w_1$ should map to $ad$, which is not an arc of $\overrightarrow{T_4}$, a contradiction.

Our cactus graph is then obtain from the circuit $v_1,\cdots,v_{g'}$ and $g'$ copies of $C$ by identifying every vertex $v_i$ with the vertex $w_1$ of a copy of $C$.
This cactus graph does not map to $\overrightarrow{T_4}$ since one of the $v_i$ would have to map to $d$ and then the copy of $C$ attached to $v_i$ would not be $\overrightarrow{T_4}$-colorable.

\section{Proof of Theorem~\ref{thm:negative}.2}
We construct a 2-edge-colored bipartite outerplanar graph with girth at least $g$ that does not map to a 2-edge-colored planar graph with at most 5 vertices.
Let $g'$ be such that $g'\ge g$ and $g'\equiv2\pmod{4}$. Consider an alternating cycle $C=v_0,\cdots,v_{g'-1}$.
For every $0\le i\le g'-3$, we add $g'-2$ 2-vertices $w_{i,1},\cdots,w_{i,g'-2}$ that form the path $P_i=v_iw_{i,1}\cdots w_{i,g'-2}v_{i+1}$
such that the edges of $P_i$ get the color distinct from the color of the edge $v_iv_{i+1}$. Let $G$ be the obtained graph.
The 2-edge-colored chromatic number of $C$ is 5.
So without loss of generality, we assume for contradiction that $G$ admits a homomorphism $\varphi$ to a 2-edge-colored planar graph $H$ on 5 vertices.
Let us define $\mathcal{E}=\bigcup_{i\texttt{ even}}\varphi(v_i)$ and $\mathcal{O}=\bigcup_{i\texttt{ odd}}\varphi(v_i)$.
Since $C$ is alternating, $\varphi(v_i)\ne\varphi(v_{i+2})$ (indices are modulo $g'$). Since $g'\equiv2\pmod{4}$, there is an odd number of $v_i$ with an even (resp. odd) index.
Thus, $\abs{\mathcal{E}}\ge3$ and $\abs{\mathcal{O}}\ge3$. Therefore we must have $\mathcal{E}\cap\mathcal{O}\ne\emptyset$.

Notice that every two vertices $v_i$ and $v_j$ in $G$ are joined by a blue path and a red path such that the lengths of these paths have the same parity as $i-j$.
Thus, the blue (resp. red) edges of $H$ must induce a connected spanning subgraph of $H$. Since $|V(H)|=5$, $H$ contains at least 4 blue (resp. red) edges.
Since red and blue edges play symmetric roles in $G$ and since $|E(H)|\le9$ by Proposition~\ref{prop:3n-6}, we assume without loss of generality that $H$ contains exactly 4 blue edges.
Moreover, these 4 blue edges induce a tree. In particular, the blue edges induce a bipartite graph which partitions $V(H)$ into 2 parts.
Thus, every $v_i$ with even index is mapped into one part of $V(H)$ and every $v_i$ with odd index is mapped into the other part of $V(H)$.
So $\mathcal{E}\cap\mathcal{O}=\emptyset$, which is a contradiction.

\section{Proof of Theorem~\ref{thm:Pmn}}
Let $T$ be a $P_g^{(m, n)}$-universal planar graph for some $g$ that is minimal with respect to the subgraph order.

By minimality of $T$, there exists a graph $G \in P_g^{(m, n)}$ such that every color in $T$ has to be used at least once to color $G$.
Without loss of generality, $G$ is connected, since otherwise we can replace $G$ by the connected graph obtained from $G$
by choosing a vertex in each component of $G$ and identifying them. We create a graph $G'$ from $G$ as follows:

For each edge or arc $uv$ we create $4m+n$ paths starting at $u$ and ending at $v$ made of vertices of degree 2:

\begin{itemize}
\item For each type of edge, we create a path made of $g-1$ edges of this type.

\item For each type of arc, we create two paths made of $g-1$ arcs of this type such that the paths alternate between forward and backward arcs.
We make the paths such that $u$ is the tail of the first arc of one path and the head of the first arc of the other path.

\item Similarly, for each type of arc we create two paths made of $g$ arcs of this type such that the paths alternate between forward and backward arcs.
We make the paths such that $u$ is the tail of the first arc of one path and the head of the first arc of the other path.
\end{itemize}

Notice that $G'$ is in $P_g^{(m, n)}$ and thus admits a homomorphism $\varphi$ to $T$.
Since $G$ is connected and every color in $T$ has to be used at least once to color $G$, we can find for each pair of vertices and $(c_1, c_2)$ in $T$
and each type of edge a path $(v_1, v_2,\cdots, v_l)$ in $G'$ made only of edges of this type such that $\varphi(v_1)=c_1$ and $\varphi(v_l)=c_2$. \newline

This implies that for every pair of vertices $(c_1, c_2)$ in $T$ and each type of edge, there exists a walk from $c_1$ to $c_2$ made of edges of this type.
Therefore, for $1\le j\le n$, the subgraph induced by $E_j(T)$ is connected and contains all the vertices of $T$.
%Since $T$ must also be able to color an odd cycle of girth at least $g$ made of one type of edge, this subgraph must also contain at least one odd cycle.
So $E_j(T)$ contains a spanning tree of $T$. Thus $T$ contains at least $|V(T)|-1$ edges of each type.\newline

Similarly, we can find for each pair of vertices $(c_1, c_2)$ in $T$ and each type of arc a path of even length $(v_1, v_2,\cdots, v_{2l-1})$ in $G'$ made only of arcs of this type,
starting with a forward arc and alternating between forward and backward arcs such that $\varphi(v_1)=c_i$ and $\varphi(v_l)=c_2$.
We can also find a path of the same kind with odd length.\newline

This implies that for every pair of vertices $(c_1, c_2)$ in $T$ and each type of arc there exist a walk of odd length and a walk of even length
from $c_1$ to $c_2$ made of arcs of this type, starting with a forward arc and alternating between forward and backward arcs.
Let $p$ be the maximum of the length of all these paths. Given one of these walks of length $l$, we can also find a walk of length $l+2$
that satisfies the same constraints by going through the last arc of the walk twice more.
Therefore, for every $l\ge p$, every pair of vertices $(c_1, c_2)$ in $T$, and every type of arc,
it is possible to find a homomorphism from the path $P$ of length $l$ made of arcs of this type, starting with a forward arc and alternating
between forward and backward arcs to $T$ such that the first vertex is colored in $c_1$ and the last vertex is colored in $c_2$.\newline

We now show that this implies that $|A_j(T)|\ge2|V(T)|-1$ for $1\le j\le m$.
Let $P$ be a path $(v_1, v_2,\cdots, v_p, v_{p+1})$ of length $p$ starting with a forward arc and alternating between forward and backward arcs of the same type.
We color $v_1$ in some vertex $c$ of $T$. Let $C_i$ be the set of colors in which vertex $v_i$ could be colored.
We know that $C_1=c$ and $C_2$ is the set of direct successors of $c$. Set $C_3$ is the set of direct predecessors of vertices in $C_2$ so $C_1\subseteq C_3$ and,
more generally, $C_i \subseteq C_i+2$. Let $uv$ be an arc in $T$. If $u\in C_i$ with $i$ odd, then $v\in C_{i+1}$.
If $v\in C_i$ with $i$ even then $u\in C_{i+1}$. We can see that $uv$ is capable of adding at most one vertex to a $C_i$ (and every $C_j$ with $j\equiv i\mod 2$ and $i\le j$).
We know that $C_{p+1}=V(T)$ hence $T$ contains at least $2|V(T)|-1$ arcs of each type.\newline

Therefore, the underlying graph of $T$ contains at least $m\paren{2|V(T)|-1}+n\paren{|V(T)|-1}=\paren{2m+n}|V(T)|-m-n$ edges, which contradicts Proposition~\ref{prop:3n-6} for $2m+n\ge3$.

\section{Proof of Theorem~\ref{thm:ce}.1}
We construct an oriented bipartite 2-outerplanar graph with girth $14$ that does not map to $\overrightarrow{T_5}$.

The oriented graph $X$ is a cycle on 14 vertices $v_0,\cdots,v_{13}$ such that the tail of every arc is the vertex with even index, except for the arc $\overrightarrow{v_{13}v_0}$.
Suppose for contradiction that $X$ has a $\overrightarrow{T_5}$-coloring $h$ such that no vertex with even index maps to $b$.
The directed path $v_{12}v_{13}v_0$ implies that $h(v_{12})\ne h(v_0)$.
If $h(v_0)=a$, then $h(v_1)\in\acc{b,c}$ and $h(v_2)=a$ since $h(v_2)\ne b$. By contagion, $h(v_0)=h(v_2)=\cdots=h(v_{12})=a$, which is a contradiction. Thus $h(v_0)\ne a$.
If $h(v_0)=c$, then $h(v_1)=d$ and $h(v_2)=c$ since $h(v_2)\ne b$. By contagion, $h(v_0)=h(v_2)=\cdots=h(v_{12})=c$, which is a contradiction. Thus $h(v_0)\ne c$.
So $h(v_0)\not\in\acc{a,b,c}$, that is, $h(v_0)\in\acc{d,e}$. Similarly, $h(v_{12})\in\acc{d,e}$.
Notice that $\overrightarrow{T_5}$ does not contain a directed path $xyz$ such that $x$ and $z$ belong to $\acc{d,e}$.
So the path $v_{12}v_{13}v_0$ cannot be mapped to $\overrightarrow{T_5}$.
Thus $X$ does not have a $\overrightarrow{T_5}$-coloring $h$ such that no vertex with even index maps to $b$.

Consider now the path $P$ on 7 vertices $p_0,\cdots,p_6$ with the arcs $\overrightarrow{p_1p_0}$, $\overrightarrow{p_1p_2}$, $\overrightarrow{p_3p_2}$, $\overrightarrow{p_4p_3}$, $\overrightarrow{p_5p_4}$, $\overrightarrow{p_5p_6}$. It is easy to check that there exists no $\overrightarrow{T_5}$-coloring $h$ of $P$ such that $h(p_0)=h(p_6)=b$.

We construct the graph $Y$ as follows: we take 8 copies of $X$ called $X_{\texttt{main}}$, $X_0$, $X_2$, $X_4$, $\cdots$, $X_{12}$.
For every couple $(i,j)\in\acc{0,2,4,6,8,10,12}^2$, we take a copy $P_{i,j}$ of $P$, we identify the vertex $p_0$ of $P_{i,j}$
with the vertex $v_i$ of $X_{\texttt{main}}$ and we identify the vertex $p_6$ of $P_{i,j}$ with the vertex $v_j$ of $H_i$.

So $Y$ is our oriented bipartite 2-outerplanar graph with girth $14$. Suppose for contradiction that $Y$ has a $\overrightarrow{T_5}$-coloring $h$.
By previous discussion, there exists $i\in\acc{0,2,4,6,8,10,12}$ such that the vertex $v_i$ of $X_{\texttt{main}}$ maps to $b$.
Also, there exists $j\in\acc{0,2,4,6,8,10,12}$ such that the vertex $v_j$ of $X_i$ maps to $b$.
So the corresponding path $P_{i,j}$ is such that $h(p_0)=h(p_6)=b$, a contradiction. Thus $Y$ does not map to $\overrightarrow{T_5}$.

\section{Proof of Theorem~\ref{thm:ce}.2}
We construct a 2-edge-colored 2-outerplanar graph with girth $11$ that does not map to $T_6$.
We take 12 copies $X_0,\cdots,X_{11}$ of a cycle of length $11$ such that every edge is red.
Let $v_{i,j}$ denote the $j^{\text{\tiny th}}$ vertex of $X_i$.
For every $0\le i\le 10$ and $0\le j\le 10$, we add a path consisting of 5 blue edges between $v_{i,11}$ and $v_{j,i}$.

Notice that in any $T_6$-coloring of a red odd cycle, one vertex must map to $c$.
So we suppose without loss of generality that $v_{0,11}$ maps to $c$.
We also suppose without loss of generality that $v_{0,0}$ maps to $c$.
The blue path between $v_{0,11}$ and $v_{0,0}$ should map to a blue walk of length 5 from $c$ to $c$ in $T_6$.
Since $T_6$ contains no such walk, our graph does not map to $T_6$.

\section{Proof of Theorem~\ref{thm:ce}.3}
We construct a 2-edge-colored bipartite 2-outerplanar graph with girth $10$ that does not map to $T_6$.
By Theorem~\ref{thm:negative}.2, there exists a bipartite outerplanar graph $M$ with girth at least $10$
such that for every $T_6$-coloring $h$ of $M$, there exists a vertex $v$ in $M$ such that $h(v)=c$.

Let $X$ be the graph obtained as follows. Take a main copy $Y$ of $M$.
For every vertex $v$ of $Y$, take a copy $Y_v$ of $M$. Since $Y_v$ is bipartite, let $A$ and $B$ the two independent sets of $Y_v$.
For every vertex $w$ of $A$, we add a path consisting of 5 blue edges between $v$ and $w$.
For every vertex $w$ of $B$, we add a path consisting of 4 edges colored (blue, blue, red, blue) between $v$ and $w$.

Notice that $X$ is indeed a bipartite 2-outerplanar graph with girth $10$.
We have seen in the previous proof that $T_6$ contains no blue walk of length 5 from $c$ to $c$.
We also check that $T_6$ contains no walk of length 4 colored (blue, blue, red, blue) from $c$ to $c$.
By the property of $M$, for every $T_6$-coloring $h$ of $X$, there exist a vertex $v$ in $Y$ and a vertex $w$ in $Y_v$ such that $h(v)=h(w)=c$.
Then $h$ cannot be extended to the path of length 4 or 5 between $v$ and $w$.
So $X$ does not map to $T_6$.

\section{Proof of Theorem~\ref{thm:NPC}.1}
Let $g$ be the largest integer such that there exists a graph in $P_g^{(1,0)}$ that does not map to $\overrightarrow{T_5}$.
Let $G\in P_g^{(1,0)}$ be a graph that does not map to $\overrightarrow{T_5}$ and such that the underlying graph of $G$ is minimal with respect to the homomorphism order.

Let $G'$ be obtained from $G$ by removing an arbitrary arc $v_0v_3$ and adding two vertices $v_1$ and $v_2$ and the arcs $v_0v_1$, $v_2v_1$, $v_2v_3$.
By minimality, $G'$ admits a homomorphism $\varphi$ to $\overrightarrow{T_5}$. Suppose for contradiction that $\varphi(v_2)=c$. This implies that $\varphi(v_1)=\varphi(v_3)=d$.
Thus $\varphi$ provides a $\overrightarrow{T_5}$-coloring of $G$, a contradiction. So $\varphi(v_2)\ne c$ and, similarly, $\varphi(v_2)\ne e$.

Given a set $S$ of vertices of $\overrightarrow{T_5}$, we say that we force $S$ if we specify a graph $H$ and a vertex $v\in V(H)$ such that 
for every vertex $x\in V\paren{\overrightarrow{T_5}}$, we have $x\in S$ if and only if there exists a $\overrightarrow{T_5}$-coloring $\varphi$ of $H$ such that $\varphi(v)=x$.
Thus, with the graph $G'$ and the vertex $v_2$, we force a non-empty set $\mathcal{S}\subset V\paren{\overrightarrow{T_5}}\setminus\acc{c,e}=\acc{a,b,d}$.

We use a series of constructions in order to eventually force the set $\acc{a,b,c,d}$ starting from $\mathcal{S}$.
Recall that $\acc{a,b,c,d}$ induces the tournament $\overrightarrow{T_4}$.
We thus reduce $\overrightarrow{T_5}$-coloring to $\overrightarrow{T_4}$-coloring, which is NP-complete for subcubic bipartite planar graphs with any given girth~\cite{GO15}.

These constructions are summarized in the tree depicted in Figure~\ref{fig:oriented}. The vertices of this forest contain the non-empty subsets of $\acc{a,b,d}$ and a few other sets.
In this tree, an arc from $S_1$ to $S_2$ means that if we can force $S_1$, then we can force $S_2$. Every arc has a label indicating the construction that is performed.
In every case, we suppose that $S_1$ is forced on the vertex $v$ of a graph $H_1$ and we construct a graph $H_2$ that forces $S_2$ on the vertex $w$.

\begin{figure}[htpb]
\begin{center}
 \includegraphics[scale=0.8]{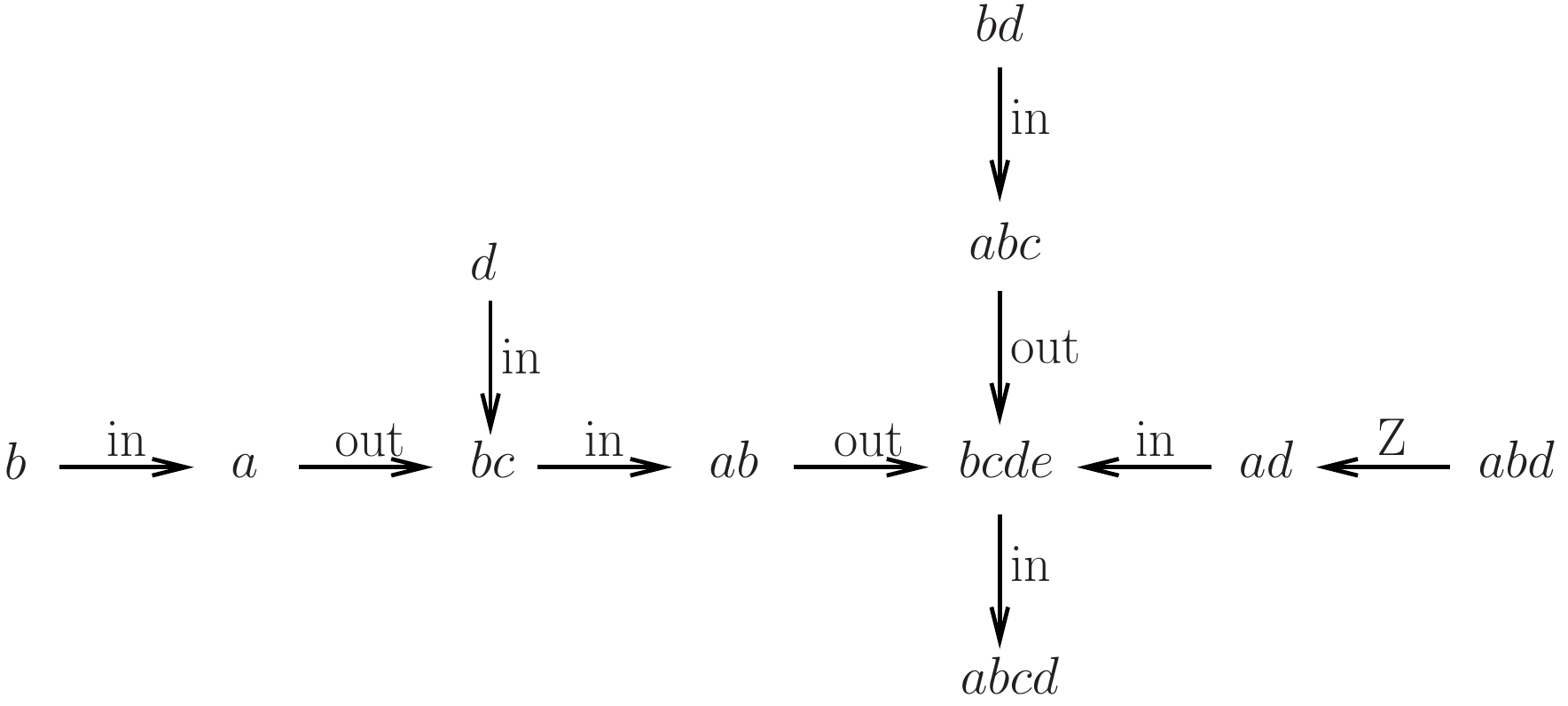}
 \caption{Forcing the set $\acc{a,b,c,d}$.\label{fig:oriented}}
\end{center}
\end{figure}

\begin{itemize}
\item Arcs labelled "out": The set $S_2$ is the out-neighborhood of $S_1$ in $\overrightarrow{T_5}$. We construct $H_2$ from $H_1$ by adding a vertex $w$ and the arc $vw$.
Thus, $S_2$ is indeed forced on the vertex $w$ of $H_2$.
\item Arcs labelled "in": The set $S_2$ is the in-neighborhood of $S_1$ in $\overrightarrow{T_5}$. We construct $H_2$ from $H_1$ by adding a vertex $w$ and the arc $wv$.
Thus, $S_2$ is indeed forced on the vertex $w$ of $H_2$.
\item Arc labelled "Z": Let $g'$ be the smallest integer such that $g'\ge g$ and $g'\equiv4\pmod{6}$. We consider a circuit $v_1,\cdots,v_{g'}$.
For $2\le i\le g'$, we take a copy of $H_1$ and we identify its vertex $v$ with $v_i$. We thus obtain the graph $H_2$ and we set $w=v_2$. Let $\varphi$ be any $T_6$-coloring of $H_2$.
By construction, $\acc{\varphi(v_2),\cdots,\varphi(v_{g'})}\subset S_1=\acc{a,b,d}$.
A circuit of length $\not\equiv0\pmod{3}$ cannot map to the 3-circuit induced by $\acc{a,b,d}$, so $\varphi(v_1)\in\acc{c,e}$.
If $\varphi(v_1)=c$ then $\varphi(v_2)=d$ and if $\varphi(v_1)=e$ then $\varphi(v_2)=a$. Thus $S_2=\acc{ad}$.
\end{itemize}

\section{Proof of Theorem~\ref{thm:NPC}.2}
Let $g$ be the largest integer such that there exists a graph in $P_g^{(0,2)}$ that does not map to $T_6$.
Let $G\in P_g^{(0,2)}$ be a graph that does not map to $T_6$ and such that the underlying graph of $G$ is minimal with respect to the homomorphism order.

Let $G'$ be obtained from $G$ by subdividing an arbitrary edge $v_0v_3$ twice to create the path $v_0v_1v_2v_3$
such that the edges $v_0v_1$ and $v_1v_2$ are red and the edge $v_2v_3$ gets the color of the original edge $v_0v_3$.
By minimality, $G'$ admits a homomorphism $\varphi$ to $T_6$.
Suppose for contradiction that $\varphi(v_1)=f$. This implies that $\varphi(v_0)=\varphi(v_2)=b$. Thus $\varphi$ provides a $T_6$-coloring of $G$, a contradiction.

Given a set $S$ of vertices of $T_6$, we say that we force $S$ if we specify a graph $H$ and a vertex $v\in V(H)$ such that 
for every vertex $x\in V(T_6)$, we have $x\in S$ if and only if there exists $T_6$-coloring $\varphi$ of $H$ such that $\varphi(v)=x$.
Thus, with the graph $G'$ and the vertex $v_1$, we force a non-empty set $\mathcal{S}\subset V(T_6)\setminus\acc{f}=\acc{a,b,c,d,e}$.

Recall that the core of a graph is the smallest subgraph which is also a homomorphic image.
We say that a subset $S$ of $V(T_6)$ is \emph{good} if the core of the subgraph induced by $S$ is isomorphic
to the graph $T_4$ which is a a clique on 4 vertices such that both the red and the blue edges induce a path of length $3$.
We use a series of constructions in order to eventually force a good set starting from $\mathcal{S}$.
We thus reduce $T_6$-coloring to $T_4$-coloring, which is NP-complete for subcubic bipartite planar graphs with any given girth~\cite{MO17}.

These constructions are summarized in the forest depicted in Figure~\ref{fig:2edge}.
The vertices of this forest are the non-empty subsets of $\acc{a,b,c,d,e}$ together with a few auxiliary sets of vertices containing $f$.
In this forest, an arc from $S_1$ to $S_2$ means that if we can force $S_1$, then we can force $S_2$. Every set with no outgoing arc is good.
We detail below the construction that is performed for each arc. In every case, we suppose that $S_1$ is forced on the vertex $v$ of a graph $H_1$
and we construct a graph $H_2$ that forces $S_2$ on the vertex $w$.

\begin{figure}
\begin{center}
 \includegraphics[scale=0.8]{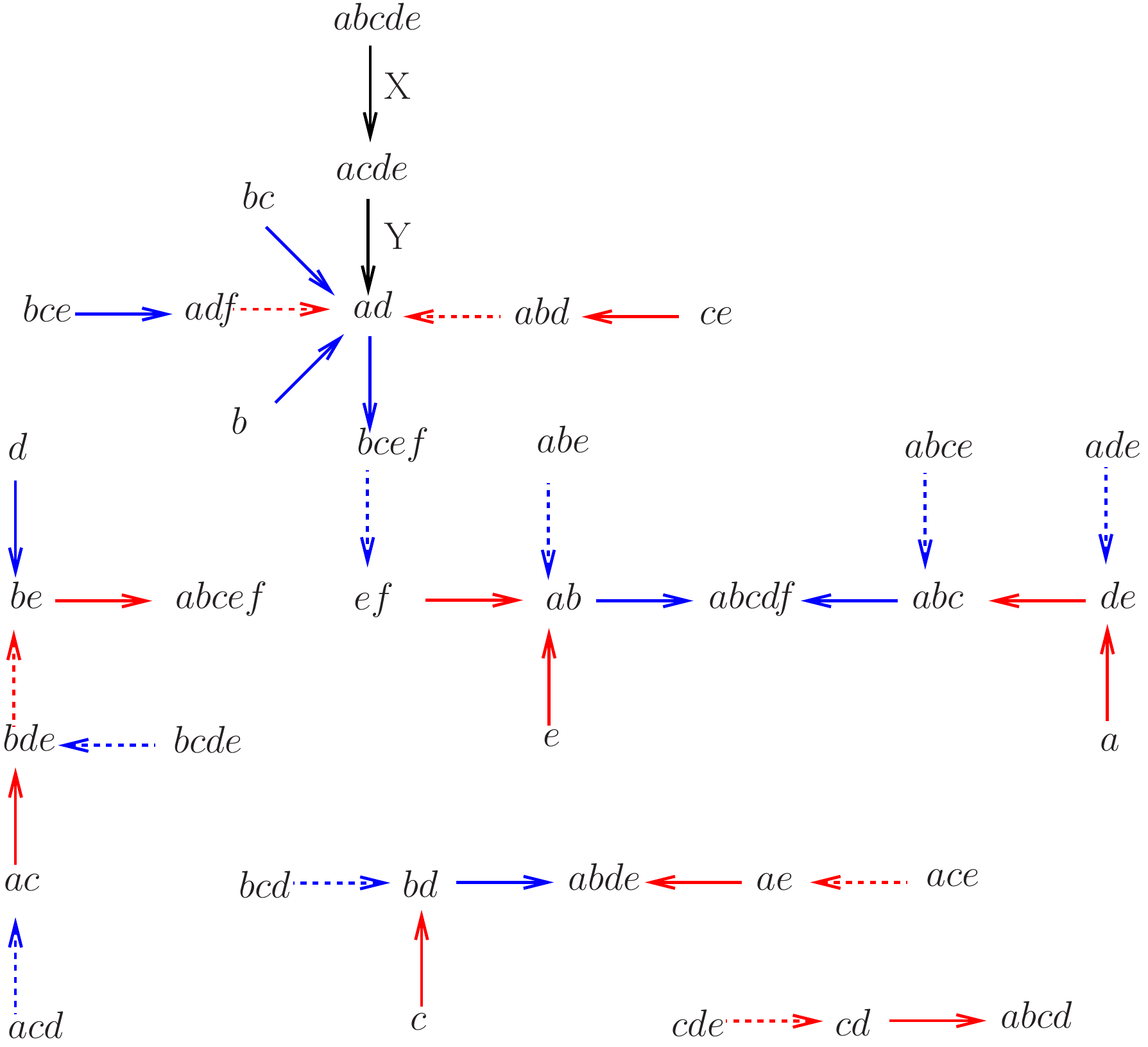}
 \caption{Forcing a good set.\label{fig:2edge}}
\end{center}
\end{figure}

\begin{itemize}
\item Blue arcs: The set $S_2$ is the blue neighborhood of $S_1$ in $T_6$. We construct $H_2$ from $H_1$ by adding a vertex $w$ adjacent to $v$ such that $vw$ is blue.
Thus, $S_2$ is indeed forced on the vertex $w$ of $H_2$.
\item Red arcs: The set $S_2$ is the red neighborhood of $S_1$ in $T_6$. The construction is as above except that the edge $vw$ is red.
\item Dashed blue arcs: The set $S_2$ is the set of vertices incident to a blue edge contained in the subgraph induced by $S_1$ in $T_6$. We construct $H_2$ from two copies of
$H_1$ by adding a blue edge between the vertex $v$ of one copy and the vertex $v$ of the other copy. Then $w$ is one of the vertices $v$.
\item Dashed red arcs: The set $S_2$ is the set of vertices incident to a red edge contained in the subgraph induced by $S_1$ in $T_6$.
The construction is as above except that the added edge is red.
\item Arc labelled "X": Let $g'=2\ceil{g/2}$. We consider an even cycle $v_1,\cdots,v_{g'}$ such that $v_1v_{g'}$ is red and the other edges are blue.
For every vertex $v_i$, we take a copy of $H_1$ and we identify its vertex $v$ with $v_i$. We thus obtain the graph $H_2$ and we set $w=v_1$.
Let $\varphi$ be any $T_6$-coloring of $H_2$. In any $T_6$-coloring of $H_2$, the cycle $v_1,\cdots,v_{g'}$ maps to a 4-cycle with exactly one red edge contained
in the subgraph of $T_6$ induced by $S_1=\acc{a,b,c,d,e}$. These 4-cycles are $aedb$ with red edge $ae$ and $cdba$ with red edge $cd$.
Since $w$ is incident to the red edge in the cycle $v_1,\cdots,v_{g'}$, $w$ can be mapped to $a$, $e$, $c$, or $d$ but not to $b$. Thus $S_2=\acc{a,c,d,e}$.
\item Arc labelled "Y": We consider an alternating cycle $v_0,\cdots,v_{8g-1}$. 
For every vertex $v_i$, we take a copy of $H_1$ and we identify its vertex $v$ with $v_i$.
We obtain the graph $H_2$ by adding the vertex $x$ adjacent to $v_0$ and $v_{4g+2}$ such that $xv_0$ and $xv_{4g+2}$ are blue. We set $w=v_0$.
In any $T_6$-coloring $\varphi$ of $H_2$, the cycle $v_1,\cdots,v_{g'}$ maps to the alternating $4$-cycle $acde$ contained in $S_1=\acc{a,c,d,e}$ such that $\varphi(v_i)=\varphi(v_{i+4\pmod{8g}})$.
So, a priori, either $\acc{\varphi(v_0),\varphi(v_{4g+2})}=\acc{a,d}$ or $\acc{\varphi(v_0),\varphi(v_{4g+2})}=\acc{c,e}$.
In the former case, we can extend $\varphi$ to $H_2$ by setting $\varphi(x)=b$. In the latter case, we cannot color $x$ since $c$ and $e$ have no common blue neighbor in $T_6$.
Thus, $\acc{\varphi(v_0),\varphi(v_{4g+2})}=\acc{a,d}$ and $S_2=\acc{a,d}$.
\end{itemize}

\bibliographystyle{alpha}
\bibliography{bib}

\begin{thebibliography}{1}

\bibitem{MNCM}
J.~Nešetřil and A.~Raspaud.
\newblock Colored homomorphisms of colored mixed graphs.
\newblock {\em Journal of Combinatorial Theory, Series B}, 80(1):147--155,
  2000.

\bibitem{grotzsch}
H.~Gr{\"o}tzsch.
\newblock Ein dreifarbensatz f{\"u}r dreikreisfreie netze auf der kugel.
\newblock {\em Wiss. Z. Martin-Luther-Univ. Halle-Wittenberg Math.-Natur.
  Reihe}, 8:109--120, 1959.

\bibitem{P10}
O.V. Borodin, A.V. Kostochka, J.~Nešetřil, A.~Raspaud, and É. Sopena.
\newblock On universal graphs for planar oriented graphs of a given girth.
\newblock {\em Discrete Mathematics}, 188(1):73--85, 1998.

\bibitem{GO15}
G.~Guegan and P.~Ochem.
\newblock Complexity dichotomy for oriented homomorphism of planar graphs with
  large girth.
\newblock {\em Theoretical Computer Science}, 596:142--148, 2015.

\bibitem{MO17}
N.~Movarraei and P.~Ochem.
\newblock Oriented, 2-edge-colored, and 2-vertex-colored homomorphisms.
\newblock {\em Information Processing Letters}, 123:42--46, 2017.

\end{thebibliography}

\end{document}